\documentclass[12pt]{amsart}
\usepackage{latexsym}
\usepackage{amsthm}
\usepackage{amsmath}
\usepackage{amsfonts}
\usepackage{amssymb}
\usepackage[dvips]{graphicx}
\usepackage{xypic}
\input xy
\xyoption{all}

\newtheorem{theo}{Theorem}[section]
\newtheorem{lemma}[theo]{Lemma}
\newtheorem{propo}[theo]{Proposition}

\newtheorem{coro}[theo]{Corollary}

\newcommand\Iso{\operatorname{Iso}}
\newcommand\op{\operatorname{op}}
\newcommand\id{\operatorname{id}}

\newcommand\Set{\operatorname{\bf Set}}
\newcommand\Gra{\operatorname{\bf Gra}}
\newcommand\Ho{\operatorname{Ho}}

\newcommand\ca{\mathcal {A}}
\newcommand\cb{\mathcal {B}}
\newcommand\cc{\mathcal {C}}
\newcommand\cd{\mathcal {D}}
\newcommand\cf{\mathcal {F}}

\newcommand\ck{\mathcal {K}}
\newcommand\cl{\mathcal {L}}
\newcommand\cm{\mathcal {M}}
\newcommand\crr{\mathcal {R}}

\newcommand\cw{\mathcal {W}}
\newcommand\cx{\mathcal {X}}

\date{June 26, 2009}
 
\begin{document}
\title[Cofibrantly generated model categories]
{Are all cofibrantly generated model categories combinatorial?}
\author[J. Rosick\'{y}]
{J. Rosick\'{y}$^*$}
\thanks{ $^*$ Supported by the Ministry of Education of the Czech republic under the project MSM 0021622409.} 
\address{\newline J. Rosick\'{y}\newline
Department of Mathematics and Statistics\newline
Masaryk University, Faculty of Sciences\newline
Kotl\' a\v rsk\' a 2, 60000 Brno, Czech Republic\newline
rosicky@math.muni.cz
}
\begin{abstract}
G. Raptis has recently proved that, assuming Vop\v enka's principle, every cofibrantly generated model category
is Quillen equivalent to a combinatorial one. His result remains true for a slightly more general concept
of a cofibrantly generated model category. We show that Vop\v enka's principle is equivalent to this claim.
The set-theoretical status of the original Raptis' result is open.
\end{abstract}

\maketitle
Combinatorial model categories were introduced by J. H. Smith as model categories which are locally presentable
and cofibrantly generated. There are of course cofibrantly generated model categories which are not combinatorial
-- the first example is the standard model category of topological spaces. This model category is Quillen equivalent
to the combinatorial model category of simplicial sets. G. Raptis \cite{R} has recently proved a somewhat surprising 
result saying that, assuming Vop\v enka's principle, every cofibrantly generated model category is Quillen equivalent 
to a combinatorial model category. Vop\v enka's principle is a set-theoretical axiom implying the existence of very 
large cardinals (see \cite{AR}). A natural question is whether Vop\v enka's principle (or other set theory) is needed 
for Raptis' result.

A \textit{model category} is a complete and cocomplete category $\cm$ together with three classes of
morphisms $\cf$, $\cc$ and $\cw$ called \textit{fibrations}, \textit{cofibrations} and \textit{weak 
equivalences} such that
\begin{enumerate}
\item[(1)] $\cw$ has the 2-out-of-3 property and is closed under retracts in the arrow category $\cm^\to$, and
\item[(2)] $(\cc,\cf\cap\cw)$ and $(\cc\cap\cw,\cf)$ are weak factorization systems.
\end{enumerate}
Morphisms from $\cf\cap\cw$ are called \textit{trivial fibrations} while morphisms from $\cc\cap\cw$
\textit{trivial cofibrations}. 

A \textit{weak factorization system} $(\cl,\crr)$ in a category $\cm$ consists of two classes $\cl$ and $\crr$ 
of morphisms of $\cm$ such that
\begin{enumerate}
\item[(1)] $\crr = \cl^{\square}$, $\cl = {}^\square \crr$, and
\item[(2)] any morphism $h$ of $\cm$ has a factorization $h=gf$ with
$f\in \cl$ and $g\in \crr$.
\end{enumerate}
Here, $\cl^{\square}$ consists of morphisms having the right lifting property w.r.t. each morphism from $\cl$ 
and ${}^\square\crr$ consists of morphisms having the left lifting property w.r.t. each morphism from $\crr$.

The standard definition of a cofibrantly generated model category (see \cite{Ho}) is that the both weak factorization 
systems from its definition are cofibrantly generated in the following sense. A weak factorization system $(\cl,\crr)$ 
is cofibrantly generated if there exists a set $\cx$ of morphisms such that 
\begin{enumerate}
\item[(1)] the domains of $\cx$ are small relative to $\cx$-cellular morphisms, and
\item[(2)] $\cx^\square=\crr$.
\end{enumerate}
Here, $\cx$-cellular morphisms are transfinite compositions of pushouts of morphisms of $\cx$. The consequence 
of this definition is that $\cl$ is the smallest cofibrantly closed class containing $\cx$. A cofibrantly closed
class is defined as a class of morphisms closed under transfinite compositions, pushouts and retracts in $\cm^\to$.
Moreover, one does not need to assume that $(\cl,\crr)$ is a weak factorization system because it follows from (1)
and (2). This observation led to the following  more general definition of a cofibrantly generated weak factorization
system (see \cite{AHRT}).

A weak factorization system $(\cl,\crr)$ is \textit{cofibrantly generated} if there exists a set $\cx$ of morphisms 
such that $\cl$ is the smallest cofibrantly closed class containing $\cx$. The consequence is that $\cx^\square=\crr$.
A model category is \textit{cofibrantly generated} if the both weak factorization systems from its definition 
are cofibrantly generated in the new sense. It does not affect the definition of a combinatorial model category
because all objects are small in a locally presentable category. Moreover, the proof of Raptis \cite{R} works
for cofibrantly generated model categories in this sense as well.

We will show that Vop\v enka's principle follows from the fact that every cofibrantly generated model category 
(in the new sense) is Quillen equivalent to a combinatorial model category. We do not know whether this is true
for standardly defined cofibrantly generated model categories as well. Our proof uses the trivial model structure
on a category $\cm$ where all morphisms are cofibrations and weak equivalences are isomorphisms.  

Given a small full subcategory $\ca$ of a category $\ck$, the \textit{canonical functor}
$$
E_\ca: \ck \to \Set^{\ca^{\op}}
$$
assigns to each object $K$ the restriction
$$
E_\ca K =\hom (-, K)\big/\ca^{op}
$$
of its hom-functor $\hom(-, K): \ck^{\op}\to \Set$ to $\ca^{\op}$ (see \cite{AR} 1.25). 

A small full subcategory $\ca$ of a category $\ck$ is called \textit{dense} provided that every object of $\ck$
is a canonical colimit of objects from $\ca$. It is equivalent to the fact that the canonical functor
$$
E_\ca: \ck \to \Set^{\ca^{\op}}
$$
is a full embedding (see \cite{AR}, 1.26). A category $\ck$ is called \textit{bounded} if it has a (small) dense 
subcategory (see \cite{AR}). 

Dense subcategories were introduced by J. R. Isbell \cite{I} and called left adequate subcategories. The following
result is easy to prove and can be found in \cite{I}.

\begin{lemma}\label{le0.1}
Let $\ca$ be dense subcategory of $\ck$ and $\cb$ a small full subcategory of $\ck$ containing $\ca$. Then $\cb$
is dense.
\end{lemma}

\begin{propo}\label{prop0.2}
Let $\ck$ be a cocomplete bounded category. Then $(\ck,\Iso)$ is a cofibrantly generated weak factorization system.
\end{propo}
\begin{proof}
Clearly, $(\ck,\Iso)$ is a weak factorization system. The canonical functor
$$
E_\ca: \ck \to \Set^{\ca^{\op}}
$$
has a left adjoint $F$ (see \cite{AR}, 1.27). The weak factorization system $(\Set^{\ca^{\op}},\Iso)$ in $\Set^{\ca^{\op}}$
is cofibrantly generated (see \cite{Ro1}, 4.6). Thus there is a small full subcategory $\cx$ of $\Set^{\ca^{\op}}$ such
that each morphism in $\Set^{\ca^{\op}}$ is a retract of a $\cx$-cellular morphism. Hence each morphism in $\ck$ is
a retract of a $F(\cx)$-cellular morphism. Thus $(\ck,\Iso)$ is cofibrantly generated.
\end{proof}

Given a complete and cocomplete category $\ck$, the choice $\cc=\ck$ and $\cw=\Iso$ yields a model category structure
on $\ck$. The corresponding two weak factorization systems are $(\ck,\Iso)$ and $(\Iso,\ck)$ and the homotopy category
$\Ho(\ck)=\ck$. We will call this model category structure \textit{trivial}. 

\begin{coro}\label{cor0.3}
Let $\ck$ be a complete, cocomplete and bounded category. Then the trivial model category structure on $\ck$ 
is cofibrantly generated.
\end{coro}
\begin{proof}
Following \ref{prop0.2}, it suffices to add that the weak factorization system $(\Iso,\ck)$ is cofibrantly generated
by $\cx=\{\id_O\}$ where $O$ is an initial object of $\ck$.
\end{proof}

\begin{theo}\label{th0.4}
Vop\v enka's principle is equivalent to the fact that every cofibrantly generated model category is Quillen equivalent
to a combinatorial model category. 
\end{theo}
\begin{proof}
Necessity follows from \cite{R}. Under the negation of Vop\v enka's principle, \cite{AR}, 6.12 presents a complete 
bounded category $\ca$ with
the following properties
\begin{enumerate}
\item[(1)] For each regular cardinal $\lambda$, there is a $\lambda$-filtered diagram $D_\lambda:\cd_\lambda\to\ck$
whose only compatible cocones $\delta_\lambda$ are trivial ones with the codomain 1 (= a terminal object),
\item[(2)] For each $\lambda$, $\id_1$ does not factorize through any component of $\delta_\lambda$.
\end{enumerate}
Since, following (1), $\delta_\lambda$ is a colimit cocone for each $\lambda$, (2) implies that $1$ is not
$\lambda$-presentable for any regular $\lambda$. Condition (2) is not stated explicitly in \cite{AR} but
it follows from the fact that there is no morphism from $1$ to a non-terminal object of $\ca$. In fact,
$\ca$ is the full subcategory of the category $\Gra$ consisting of graphs $A$ without any morphism $B_i\to A$
where $B_i$ is the rigid class of graphs indexed by ordinals (whose existence is guaranteed by the negation
of Vop\v enka's principle). The existence of a morphism $1\to A$ means the presence of a loop in $A$ and,
consequently, the existence of a constant morphism $B_i\to A$ (having a loop as its value).

Assume that the trivial model category $\ca$ is Quillen equivalent to a combinatorial model category $\cm$. 
Since $\Ho\cm$ is equivalent to $\ca$, it shares properties (1) and (2). Moreover, since $\Ho\ck=\ck$, 
the diagrams $D_\lambda$ are diagrams in $\ck$. It follows from the definition of Quillen equivalence
that the corresponding diagrams in $\Ho\cm$ (we will denote them by $D_\lambda$ as well) can be rectified.
It means that there are diagrams $\overline{D}_\lambda$ in $\cm$ such that $D_\lambda=P\overline{D}_\lambda$;
here, $P:\cm\to\Ho\cm$ is the canonical functor. Following \cite{D} and \cite{Ro}, there is a regular cardinal 
$\lambda_0$ such that the replacement functor $R:\cm\to\cm$ preserves $\lambda_0$-filtered colimits. $R$ sends 
each object $M$ to a fibrant and cofibrant object and the canonical functor $P$ can be taken as the composition 
$QR$ where $Q$ is the quotient functor identifying homotopy equivalent morphisms. 

Let 
$$
(\overline{\delta}_{\lambda d}:\overline{D}_\lambda d\to M_\lambda)_{d\in\cd_\lambda}
$$
be colimit cocones. Then
$$
(R\overline{\delta}_{\lambda d}:R\overline{D}_\lambda d\to RM_\lambda)_{d\in\cd_\lambda}
$$
are colimit cocones for each $\lambda>\lambda_0$. Following (1), $RM_\lambda\cong 1$ for each $\lambda>\lambda_0$. 
The object $RM_{\lambda_0}$ is $\mu$-presentable in $\cm$ for some regular cardinal $\lambda_0<\mu$. Since
$RM_{\lambda_0}$ and $RM_\mu$ are homotopy equivalent, there is a morphism $f:RM_{\lambda_0}\to RM_\mu$. Since
$f$ factorizes through some $R\overline{\delta}_{\mu d}$, $\id_1$ factorizes through some component 
of $\delta_\mu$, which contradicts (2).
\end{proof}

While the weak factorization system $(\Iso,\ck)$ is cofibrantly generated in the sense of \cite{Ho}, it is not true
for $(\ck,\Iso)$ because the complete, cocomplete and bounded category in \cite{AR}, 6.12 is not locally
presentable just because it contains a non-presentable object. Thus we do not know whether Vop\v enka's principle
follows from the original result from \cite{R}.  

The proof above does not exclude that $\ca$ \textit{has a combinatorial model}, i.e., that there is a combinatorial
model category $\cm$ such that $\ca$ is equivalent to $\Ho\cm$.

\begin{propo}\label{prop0.4} Assume the existence of a proper class of compact cardinals and let $\ck$ be a complete, 
cocomplete and bounded category. Then the trivial model category $\ck$ has a combinatorial model if and only if $\ck$ 
is locally presentable.
\end{propo}
\begin{proof}
If $\ck$ is locally presentable the trivial model category $\ck$ is combinatorial. Assume that the trivial model 
category $\ck$ is equivalent to $\Ho\cm$ where $\cm$ is a combinatorial model category. Let $\cx$ be a dense 
subcategory of $\ck$. Following \cite{Ro}, 4.1, there is a regular cardinal $\lambda$ such that 
\begin{enumerate}
\item[(1)] $\cx\subseteq P(\cm_\lambda)$ where $\cm_\lambda$ denotes the full subcategory of $\cm$ consisting
of $\lambda$-presentable objects,
\item[(2)] The composition $H=E_{P(\cm_\lambda)}\cdot P$ preserves $\lambda$-filtered colimits.
\end{enumerate}
Since $P(\cm_\lambda)$ is dense in $\ck$ (see \ref{le0.1}), $E_{P(\cm_\lambda)}$ is a full embedding. Hence $\ck$ is 
the full image of the functor $H$, i.e., the full subcategory on objects $H(M)$ with $M$ in $\cm$. Following
\cite{Ro2}, Corollary of Theorem 2, $\ck$ is locally presentable.
\end{proof}

Vop\v enka's principle is stronger than the existence of a proper class of compact cardinals. Thus, assuming
the negation of Vop\v enka's principle but the existence of a proper class of compact cardinals, there is
a cofibrantly generated model category without a combinatorial model.

\end{document}